\documentclass{amsart}
\usepackage[utf8]{inputenc}

\title[Steenrod closed $C_3$-invariant parameter ideals]{Steenrod closed $C_3$-invariant parameter ideals in the mod 2 cohomology of $\mathbb{Z}/2\times\mathbb{Z}/2$}

\author{Henrik Rüping and Marc Stephan}
\date{\today}

\usepackage{tikz} 
\usepackage{amsmath}
\usetikzlibrary{cd}
\usetikzlibrary{patterns,decorations}
\usepackage[utf8]{inputenc}
\usepackage{amsmath, amssymb}
\usepackage{stmaryrd}
\usepackage{mathtools}
\usepackage[mathcal]{euscript}
\usepackage{tikz}
\usepackage{dsfont}
\usepackage{float}
\usepackage{enumitem}
\usetikzlibrary{matrix}
\usepackage[all]{xy}
\definecolor{darkblue}{rgb}{0,0,0.6}
\usepackage[ocgcolorlinks,colorlinks=true, citecolor=darkblue, filecolor=darkblue, linkcolor=darkblue, urlcolor=darkblue]{hyperref}
\usepackage[capitalize,noabbrev]{cleveref}

\newcounter{commentcounter}

\newcommand{\ignore}[1]{}

\newcommand{\FF}{\mathbb{F}}
\newcommand{\ZZ}{\mathbb{Z}}

\DeclareMathOperator{\Sq}{Sq}
\DeclareMathOperator{\SL}{SL}

\DeclareMathOperator{\Ext}{Ext}

\DeclareMathOperator{\Qd}{Qd}

\numberwithin{equation}{section}
\newtheorem{theorem}[equation]{Theorem}

\newtheorem{corollary}[equation]{Corollary}
\newtheorem{lemma}[equation]{Lemma}
\newtheorem{proposition}[equation]{Proposition}

\theoremstyle{definition}
\newtheorem{example}[equation]{Example}
\newtheorem{definition}[equation]{Definition}

\newtheorem{remark}[equation]{Remark}
\newtheorem{question}[equation]{Question}



\begin{document}
\begin{abstract}
For the nontrivial action by the cyclic group  $C_3$ of order $3$ on the graded polynomial ring $\FF_2[a,b]$, we classify the $C_3$-invariant parameter ideals that are closed under Steenrod operations. The classification has applications to free actions by the Klein four-group $\ZZ/2\times \ZZ/2$ on products of two spheres (and more generally, finite CW complexes with four-dimensional mod $2$ homology) that extend to actions by the alternating group $A_4=(\ZZ/2\times \ZZ/2)\rtimes C_3$.
\end{abstract}
\maketitle

\section{Introduction}

A classical topic in equivariant topology concerns actions of finite groups on products of spheres. For example, it is open whether the group $\Qd(p)=(\ZZ/p\times \ZZ/p)\rtimes \SL_2(\FF_p)$ for odd primes $p$ acts freely on a finite CW complex with homotopy type a product of two spheres, i.e., whether it satisfies the rank conjecture of Benson and Carlson \cite{bensoncarlson1987}. As already apparent in the work of Benson and Carlson, it is fruitful to study equivariant complexes both algebraically and topologically.

The alternating group on four letters $A_4$ is a semi-direct product $A_4=(\ZZ/2\times \ZZ/2)\rtimes C_3$ and is similar to $\Qd(p)$, but less complicated. Neither $A_4$ nor $\Qd(p)$ can act freely on a product of two equi-dimensional spheres, see \cite[Theorem 2]{oliver1979} and \cite[Theorem~1.3]{okayyalcin2018} respectively. Inspired by Oliver's result, we investigated finite, free $A_4$-CW complexes with four-dimensional $\FF_2$-cohomology topologically in \cite{ruepingstephanyalcin2022} and perfect cochain complexes over the group ring $\FF_2[A_4]$ with four-dimensional homology algebraically in \cite{ruepingstephan2024}. Algebraically, there is a complete classification; see \cite[Theorem~1.1]{ruepingstephan2024}: The perfect cochain complexes over $\FF_2[A_4]$ with four-dimensional homology are classified by an integer (corresponding to the degree of the lowest nonzero homology group) and a $C_3$-invariant parameter ideal in the group cohomology of $P\coloneqq\ZZ/2\times \ZZ/2$.

Steenrod operations provide an obstruction to realizing perfect cochain complexes over $\FF_2[A_4]$ topologically. If the perfect cochain complex with four-dimen{\-}sional homology arises from a finite CW complex with a free $A_4$-action (or more generally, $P$-free $A_4$-action), then the corresponding parameter ideal is closed under Steenrod operations; see \cite[Proposition~7.4]{ruepingstephanyalcin2022},
\cite[Lemma~8.1]{ruepingstephan2024}.

In joint work with Erg{\" u}n Yal{\c c}{\i }n, we classified the Steenrod closed parameter ideals
in $H^*(BA_4)\cong H^*(BP)^{C_3}$; see \cite[Theorem~1.2]{ruepingstephanyalcin2022}. Any such ideal extends to a Steenrod closed $C_3$-invariant parameter ideal in $H^*(BP)$. Here we complete the classification: The remaining Steenrod closed $C_3$-invariant parameter ideals in $H^*(BP)\cong \FF_2[a,b]$ are $\langle a^i, b^i\rangle$ where $i$ is any power of $2$; see \cref{thm:classification_ideals_from_orbits} and \cref{prop:crit_notinv_parameters}.

This has topological applications: 
\begin{theorem}[see \cref{thm:actionspaces} and \cref{thm:oddaction0137}]\label{thm:actionspacesintro}
Let $P=\ZZ/2\times \ZZ/2$ be the $2$-Sylow subgroup of $A_4$.
\begin{enumerate}[wide, labelindent=0pt]
\item\label{item:Borelcohomology_determined} If $X$ is a  finite, free $P$-CW complex such that $H^*(X;\FF_2)\cong H^*(S^n\times S^n;\FF_2)$ for some $n\geq 0$ as graded vector spaces and the $P$-action extends to an $A_4$-action, then the $P$-equivariant Borel cohomology ring of $X$ is
\[H^*(X/P;\FF_2)\cong \FF_2[a,b]/\langle a^{n+1}, b^{n+1}\rangle\,.\]
\item\label{item:actiononproductspheres} There exists a finite, free $P$-CW complex of the homotopy type of $S^n\times S^n$ such that the $P$-action extends to an $A_4$-action if and only if $n\in \{0,1,3,7\}$.
\end{enumerate}
\end{theorem}

The construction of the $P$-free $A_4$-actions on $S^n\times S^n$ in \cref{thm:actionspacesintro}\eqref{item:actiononproductspheres} is due to \cite{plaktha1996}. The necessity of $n\in \{0,1,3,7\}$ follows from \cite{schultz1971,adembrowder1988} and the full statement of \cref{thm:actionspaces}.

For general finite $A_4$-CW complexes with four-dimensional mod $2$ cohomology such that the restricted $P$-action is free, we provide extensive restrictions for the possible degrees of the nonzero cohomology groups in \cref{thm:degreesofparameters}. Similarly to the restrictions for free $A_4$-actions on products of two spheres from \cite[Theorem~1.4]{ruepingstephanyalcin2022}, we show that almost no degrees can occur in the following sense:

 \begin{theorem}[\cref{cor:percentageofpairs}]
    For $r\ge 0$, the percentage of those pairs $(m,n)$ with $m,n\le r$ such that there exists a finite $P$-free, $A_4$-CW complex whose mod $2$ cohomology is four-dimensional with homogeneous basis elements in degrees $0,m,n,d$ for some $d\ge m,n$ tends to zero as $r\to \infty$.
\end{theorem}
 
The dimensions of the occurring nonzero homology groups in \cref{thm:actionspacesintro}\eqref{item:Borelcohomology_determined} not only determine the $P$-equivariant Borel cohomology of $X$, but also the homotopy type of its cellular chain complex over $\FF_2[A_4]$.
\begin{theorem}[\cref{thm:rigid}]\label{thm:intro:rigid}
Let $X$ and $Y$ be two finite $A_4$-CW complexes with four-dimensional cohomologies over $\FF_2$ such that $H^*(X;\FF_2)\cong H^*(Y;\FF_2)$ as graded vector spaces. If the restricted $P$-actions on $X$ and $Y$ are free, then the cellular cochain complexes of $X$ and $Y$ are homotopy equivalent over $\FF_2[A_4]$. 
\end{theorem}
\cref{thm:intro:rigid} generalizes \cite[Theorem~8.2]{ruepingstephan2024} from free $A_4$-actions to $P$-free $A_4$-actions.

\subsection*{Acknowledgments}
The research of Stephan was partially funded by the Deutsche Forschungsgemeinschaft (DFG, German Research Foundation) – Project-ID\\
491392403~–~TRR~358.

\section{Notation, conventions and recollection of invariant theory}\label{sec:notation}
We take ordinary cohomology with coefficients in $\FF_2$ and often suppress the coefficients from the notation.

We fix a generator $\varphi$ of the cyclic group $C_3$ of order $3$ and let it act on the Klein four-group $P\coloneqq \ZZ/2\times \ZZ/2$ by $\varphi(x,y) = (y,x+y)$ for $(x,y)\in P$. The alternating group $A_4$ is identified with the semidirect product $P\rtimes C_3$. For any finite group $Q$ acting on $P$, conjugation induces a $Q$-action on $H^*(BP)=\Ext^*_{\FF_2[P]}(\FF_2,\FF_2)$; see e.g. \cite[Section~2]{ruepingstephan2024}. An ideal in $H^*(BP)$ is \emph{$Q$-invariant}, if it closed under the $Q$-action.
For $Q=C_3$ and 
\[H^*(BP)\cong H^*(B\ZZ/2)\otimes_{\FF_2}H^*(B\ZZ/2)\cong \FF_2[a]\otimes_{\FF_2}[b]\cong \FF_2[a,b]\]
with $|a|=|b|=1$, the $C_3$-action on $H^*(BP)$ is determined by $\varphi(a)= b$ and $\varphi(b)=a+b$. 

We use the setting of \cite[Section~2]{ruepingstephanyalcin2022}. The group cohomology $H^*(BA_4)$ can be calculated as the invariant ring $H^*(BP)^{C_3}$; see \cite[Chapter~III, Theorem~1.3]{adem2013cohomology}, and is a unique factorization domain; see \cite[Theorem~2.11]{nakajima1982}. A homogeneous ideal $J\subset H^*(BP)$ or $J\subset H^*(BA_4)$ is a \emph{parameter ideal} if it is generated by two homogeneous elements of positive degrees such that $H^*(BP)/J$ is finite over $\FF_2$. In this case, we say the two generators form a system of parameters. In our situation, two homogeneous elements $x_1$, $x_2$ of positive degrees form a system of parameters if and only if they satisfy the following equivalent conditions by \cite[Lemma~2.6]{ruepingstephanyalcin2022}:
\begin{enumerate}
    \item $x_1$, $x_2$ is a regular sequence,
    \item $x_1$, $x_2$ are coprime.
\end{enumerate}

With an action of a finite group $G$ on a CW complex $X$, we mean that $X$ admits the structure of a $G$-CW complex.

\section{Two families of \texorpdfstring{$C_3$}{C3}-invariant parameter ideals}
We will provide equivalent criteria to determine whether a $C_3$-invariant parameter ideal of $H^*(BP)$ is an extension of a parameter ideal in $H^*(BA_4)=H^*(BP)^{C_3}$. Interest in these ideals arises from the following two results.

\begin{proposition}[{\cite[Proposition~4.3(1),(4) and Lemma~8.1]{ruepingstephan2024}}]\label{prop:Steenrodclosedparameteridealfromaction}
    Suppose that $A_4$ acts on a finite CW complex $X$ with four-dimensional total cohomology such that $P$ acts freely.
    Then the induced map $H^*(BP) \to H^*(X/P)$ is surjective and its kernel is a $C_3$-invariant Steenrod closed parameter ideal.
  
    Moreover, if $0\leq m\leq n\leq d$ are the degrees of four homogeneous basis elements of $H^*(X)$, then the parameters of the $C_3$-invariant Steenrod closed parameter ideal have degrees $m+1$, $n+1$.
\end{proposition}

Algebraically, the $C_3$-invariant parameter ideals classify the perfect cochain complexes over $\FF_2[A_4]$ with four-dimensional homology up to shifting.

\begin{theorem}[{\cite[Theorem~1.1]{ruepingstephan2024}}]\label{thm:classification}
    Let $Q$ be a finite group of odd order acting on $P=\ZZ/2\times \ZZ/2$ and $G=P\rtimes Q$.
    There is a bijection between quasi-isomor\-phism classes of perfect cochain complexes over $\FF_2[G]$ with four-dimensional homology and pairs $(l,J)$ of an integer $l$ and a $Q$-invariant parameter ideal $J$.
\end{theorem}
The bijection assigns to a perfect cochain complex $C^*$ with $\dim_{\FF_2} H^*(C^*)=4$ the degree $l$ of the lowest nonzero homology group and the annihilator ideal $J\subset H^*(BP)$ of the $H^*(BP)$-module $\Ext^*_{\FF_2[P]}(\FF_2,C^*)$. We are mainly interested in the case of $Q=C_3$, but also the trivial group $Q=1$.

The following result and \cref{prop:crit_notinv_parameters} split the set of $C_3$-invariant parameter ideals in two families.
\begin{proposition}\label{prop:crit_inv_parameters}
Let $C^*$ be a perfect $\FF_2[A_4]$-cochain complex with four-dimensional homology $H^*(C^*)$ and let $J\subset H^*(BP)$ be the annihilator ideal of $\Ext^*_{\FF_2[P]}(\FF_2,C^*)$.

The following statements are equivalent:
    \begin{enumerate}
        \item \label{it:crit_inv_parametersi} $J$ has one $C_3$-invariant parameter.
        \item \label{it:crit_inv_parametersii}The graded $\FF_2[C_3]$-representation $J/\langle a,b\rangle J$ is trivial.
        \item \label{it:crit_inv_parametersiii}We can find a system of $C_3$-invariant parameters $x,y$ for $J$.
        \item \label{it:crit_inv_parametersiv} $J$ is the extension of a parameter ideal $J'\subset H^*(BA_4)=H^*(BP)^{C_3}$.        
        \item \label{it:crit_inv_parametersv}The $C_3$-action on $H^*(C^*)$ is trivial.
    \end{enumerate}
\end{proposition}
\begin{proof}
First, we want to show that \eqref{it:crit_inv_parametersi} implies \eqref{it:crit_inv_parametersii}.
Let $x$ be a $C_3$-invariant parameter of $J$. Note that the $\FF_2[C_3]$ representation $J/\langle a,b\rangle J$ is two-dimensional. Thus it is either the direct sum of two trivial irreducible representations or the non-trivial two-dimensional representation. It cannot be the latter, since $[x]\in J/\langle a,b\rangle J$ is fixed.

Next, we show that \eqref{it:crit_inv_parametersii} implies \eqref{it:crit_inv_parametersiii}. Pick representatives $x_1,x_2$ of a homogeneous basis of $J/\langle a,b\rangle J$. Applying the Reynolds operator $\mathcal R(z)\coloneqq 1/3\sum_{q \in C_3} qz$ to $x_1,x_2$ yields $C_3$-invariant representatives of the same basis elements, since the $C_3$-action on $J/\langle a,b\rangle J$ is trivial. These generate $J$ by the graded Nakayama lemma and since $J$ is a parameter ideal, they form a system of parameters.

We show that \eqref{it:crit_inv_parametersiii} implies \eqref{it:crit_inv_parametersiv}. Next, suppose that $x_1,x_2$ are two $C_3$-invariant parameters for $J$. Define $J'\subset H^*(BA_4)=H^*(BP)^{C_3}$ as the ideal generated by $x_1,x_2$. By construction, $J$ is the extension of $J'$. In particular, $J'$ is a parameter ideal by \cite[Lemma~2.3]{ruepingstephanyalcin2022}.

To show that \eqref{it:crit_inv_parametersiv} implies \eqref{it:crit_inv_parametersi}, suppose that $J$ is the extension of a parameter ideal $J'\subset H^*(BP)^{C_3}$. Any system of parameters for $J'$ is a system of $C_3$-invariant parameters for $J$.

Finally, we establish the equivalence of  \eqref{it:crit_inv_parametersv} and \eqref{it:crit_inv_parametersii}. Let $l$ be the degree of the lowest non-zero homology group of $C^*$. By \cite[Proposition~4.3]{ruepingstephan2024} and since one-dimensional representations over $\FF_2$ are trivial, the cohomology of $\Sigma^{-l}C^*$ is isomorphic to the exterior algebra $\Lambda (\Sigma^{-1} (J/\langle a,b\rangle J))$ on the shifted graded module $\Sigma^{-1} (J/\langle a,b\rangle J)$ as graded $\FF_2[C_3]$-representations. Hence, the action on homology is trivial if and only if the graded $C_3$-representation $J/\langle a,b\rangle J$ is trivial.
\end{proof}
Our goal is to classify all Steenrod closed $C_3$-invariant parameter ideals in the group cohomology $H^*(BP)$. Previous joint work with Erg\"un Yal{\c c}{\i}n provides the answer for the ideals from \cref{prop:crit_inv_parameters}.

\begin{remark}\label{rem:bijectionextensionrestriction}
    Extension and restriction induce a bijection between the set of parameter ideals $J'$ in $H^*(BA_4)$ and the set of $C_3$-invariant parameter ideals $J$ in $H^*(BP)$ for which the equivalent conditions from \cref{prop:crit_inv_parameters} hold; see \cite[Lemma~2.3]{ruepingstephanyalcin2022}. Moreover, $J'$ is closed under Steenrod operations if and only if its extension is Steenrod closed by \cite[Lemma~2.8]{ruepingstephanyalcin2022}.

    The Steenrod closed parameter ideals in $H^*(BA_4)$ have been classified in \cite[Theorem~1.2]{ruepingstephanyalcin2022}.    
\end{remark}

In \cref{sec:steenrod_closed_parameter_ideals} we will classify the remaining Steenrod closed $C_3$-invariant parameters ideals of $H^*(BP)$. These are the Steenrod closed $C_3$-invariant parameter ideals for which the conditions of \cref{prop:crit_inv_parameters} do not hold. Recall the fixed generator $\varphi$ of $C_3$ from \cref{sec:notation}.

\begin{proposition}\label{prop:crit_notinv_parameters} Let $J\subset H^*(BP)$ be a homogeneous ideal. The following statements are equivalent:
    \begin{enumerate}
        \item \label{it:crit_notinv_parametersi} $J$ is a $C_3$-invariant parameter ideal such that none of the equivalent conditions from $\cref{prop:crit_inv_parameters}$ hold.
        \item \label{it:crit_notinv_parametersii}$J$ is nonzero and for every nonzero homogeneous element $x\in J$ of smallest degree, we have $x+\varphi(x)+\varphi^2(x)=0$, the elements $x,\varphi(x)$ are coprime, and $\langle x, \varphi(x)\rangle=J$.
        \item \label{it:crit_notinv_parametersiii} $J$ is a parameter ideal and generated by the $C_3$-orbit of a homogeneous element $x\in H^*(BP)$.
    \end{enumerate}
\end{proposition}
\begin{proof}
First, we show that \eqref{it:crit_notinv_parametersi} implies \eqref{it:crit_notinv_parametersii}. By assumption, \cref{prop:crit_inv_parameters}\eqref{it:crit_inv_parametersii} does not hold. Thus, $J/\langle a,b\rangle J$ has to be the nontrivial, two-dimensional $C_3$-representation, whose only fixed element is $0$. Let $x\in J$ be a nonzero homogeneous element  of smallest degree. Then the elements $x$ and $\varphi(x)$ represent different, nonzero elements in $J/\langle a,b\rangle J$. Hence, their classes are a basis of the two-dimensional $\FF_2$-vector space $J/\langle a,b\rangle J$. By the graded Nakayama lemma, they generate $J$ and since $J$ is a parameter ideal, they are coprime; see \cite[Lemma~2.6]{ruepingstephanyalcin2022}. Every element $z$ of the nontrivial two-dimensional $C_3$-representation satisfies $z+\varphi(z)+\varphi^2(z)=0$. Hence $x+\varphi(x)+\varphi^2(x)$ is zero in $J/\langle a,b\rangle J$ and for degree reasons it is zero in $J$ as well.

Next, we show that \eqref{it:crit_notinv_parametersii} implies \eqref{it:crit_notinv_parametersiii}. Let $x\in J$ be a nonzero element of smallest degree. Then $J=\langle x,\varphi(x)\rangle =\langle C_3x\rangle$ using the assumption $x+\varphi(x)+\varphi^2(x)=0$. Since $x$ and $\varphi(x)$ are coprime, they form a system of parameters; see, e.g. \cite[Lemma~2.6]{ruepingstephanyalcin2022}.

Finally, we establish that \eqref{it:crit_notinv_parametersiii} implies \eqref{it:crit_notinv_parametersi}. Since the parameter ideal $J$ is generated by a $C_3$-orbit of $x\in H^*(BP)$, the two-dimensional representation $J/\langle a,b\rangle J$ is generated by $[x]$ as an $\FF_2[C_3]$-module and hence is not the trivial representation.
\end{proof}
\begin{example} The ideal $\langle C_3a\rangle\subset H^*(BP)$ is a $C_3$-invariant parameter ideal. Its restriction to $H^*(BA_4)$ is $H^{>0}(BA_4)$ which is not a parameter ideal.
\end{example}

\section{Steenrod closed parameter ideals}\label{sec:steenrod_closed_parameter_ideals}
The total Steenrod operation on $H^*(BP)\cong \FF_2[a,b]$ is the ring homomorphism $\Sq\colon \FF_2[a,b]\to \FF_2[a,b]$ with $\Sq(a)=a+a^2$, $\Sq(b)=b+b^2$. It restricts to the total Steenrod operation on $H^*(BA_4)$. We say that a homogeneous ideal $J\subset H^*(BP)$ or $J\subset H^*(BA_4)$ is \emph{Steenrod closed} if it is closed under the total Steenrod operation.

In this section, we classify the $C_3$-invariant Steenrod closed parameter ideals in $H^*(BP)$ that are not extensions of Steenrod closed parameter ideals in $H^*(BA_4)$.
\begin{lemma}\label{lem:parameterssamedegree} A $C_3$-invariant Steenrod closed parameter ideal $J\subset H^*(BP)$ is generated by an orbit if and only if it has parameters of the same degree. 
\end{lemma}
\begin{proof}
By \cref{prop:crit_notinv_parameters}, any parameter ideal generated by an orbit has parameters of the same degree. If a $C_3$-invariant parameter ideal $J$ is not generated by an orbit, then it is the extension of a parameter ideal $J'\subset H^*(BP)^{C_3}$ by \cref{prop:crit_notinv_parameters} and \cref{prop:crit_inv_parameters}. If in addition $J$ is Steenrod closed, then so is $J'$ by \cref{rem:bijectionextensionrestriction}. Then $J'$ can not have parameters of the same degree as Oliver showed in \cite[Lemma~1]{oliver1979} that all Steenrod closed ideals of $H^*(BP)^{C_3}$ generated by elements of the same degree are principal ideals. The parameters for $J'$ are also parameters for $J$, therefore $J$ has parameters of different degrees.
\end{proof}

We recall the (generalized) Kameko maps from \cite[Definition~3.6]{ruepingstephanyalcin2022},\cite[Definition~1.6.2]{walkerwood2018vol1}.
\begin{definition} Any element $x$ in $H^*(BP)\cong \FF_2[a,b]$ can be uniquely written in the form 
$$x=\kappa_1(x)^2+\kappa_a(x)^2a+\kappa_b(x)^2b+\kappa_{ab}(x)^2ab$$ for some elements $\kappa_1(x),\kappa_a(x),\kappa_b(x),\kappa_{ab}(x)$. The maps $\kappa_*$ for $\ast\in \{1,a,b,ab\}$ are called \emph{Kameko maps}.
\end{definition} 

\begin{remark} The Kameko maps are additive and satisfy $\kappa_*(xz^2)=\kappa_*(x)z$ for any elements $x,z\in \FF_2[a,b]$. The maps $\kappa_a$ and $\kappa_b$ vanish on homogeneous elements of even degree. The maps $\kappa_1$ and $\kappa_{ab}$ vanish on homogeneous elements of odd degree.
\end{remark}
We use the following notation.
\begin{definition}
For a subset $S\subset H^*(BP)$, we define $S^{[2]}$ to be the set $\{s^2\mid s\in S\}$.
\end{definition}

\begin{lemma}\label{lem:interestingsquared}
    An orbit $S$ in $H^*(BP)$ generates a Steenrod closed parameter ideal if and only if $S^{[2]}$ does.
\end{lemma}
\begin{proof}
We apply \cref{prop:crit_notinv_parameters} to show that an orbit $S=C_3x$ of a homogeneous element $x\in H^*(BP)$ of positive degree generates a parameter ideal if and only if $S^{[2]}$ does. In characteristic two we have $(x+\varphi(x)+\varphi^2(x))^2=x^2+\varphi(x^2)+\varphi^2(x^2)$. It follows that $x+\varphi(x)+\varphi^2(x)=0$ if and only if $x^2+\varphi(x^2)+\varphi^2(x^2)=0$, because $H^*(BP)$ has no zero-divisors. Since $\FF[a,b]$ is a unique factorization domain, the elements $x,\varphi(x)$ are coprime if and only if $x^2,\varphi(x)^2$ are coprime. 

We are left to show that the parameter ideal $\langle x, \varphi(x)\rangle$ generated by an orbit $C_3x$ is Steenrod closed if and only if the ideal $\langle x^2, \varphi(x)^2\rangle$ is Steenrod closed. Since $\langle x,\varphi(x)\rangle =\langle \varphi(x),\varphi^2(x)\rangle$, it suffices to show that $\Sq(x)\in \langle x,\varphi(x)\rangle$ if and only if $\Sq(x^2)\in \langle x^2,\varphi(x^2)\rangle$.
First, suppose that $\Sq(x) = \lambda x+\mu \varphi(x)$.
    Since the total Steenrod square and squaring are ring homomorphisms, we get
    $$\Sq(x^2)=\Sq(x)^2=(\lambda x+\mu \varphi(x))^2=\lambda^2 x^2+\mu^2 \varphi(x^2).$$ 

    Conversely, assume that $\Sq(x^2)=\alpha x^2+\beta \varphi(x^2)$.
    Apply $\kappa_1$ to obtain $\Sq(x) = \kappa_1(\alpha)x+\kappa_1(\beta)\varphi(x)$.    
\end{proof}

\begin{lemma}\label{lem:poweroftwo_orbit}
    Let $n$ be a power of two. Then the $C_3$-orbit of $a^n$ generates a Steenrod closed parameter ideal.
\end{lemma}

\begin{proof}
By \cref{lem:interestingsquared}, it suffices to consider the case of $n=1$. Then $\langle C_3 a\rangle$ is the Steenrod closed parameter ideal $\langle a, b\rangle\subset \FF_2[a,b]$.
\end{proof}

The main goal of this section is to show that these are all orbits which generate Steenrod closed parameter ideals. The following technical lemma allows us to identify a specific element of an orbit that generates a parameter ideal. 

\begin{lemma}\label{lem:wloghighestsummands}
Assume that an orbit $S\subset H^n(BP)$ generates a parameter ideal. Let $p:H^n(BP)\to \FF_2\oplus \FF_2$ be the $\FF_2$-linear map sending $\sum_i \lambda_i a^ib^{n-i}$ to $(\lambda_n,\lambda_0)$. Then there exists a unique $v\in S$ such that $p(v)=(1,0)$. For this $v$, we obtain $p(\varphi(v))=(0,1)$ and $p(\varphi^2(v))=(1,1)$.
\end{lemma}
\begin{proof}
Let $w\in S$ be an arbitrary nonzero element of $S$. Then $p(w)$ cannot be zero, since otherwise $w$ would be divisible by $ab$ and $\varphi(w)$ would be divisible by $\varphi(ab)=b(a+b)$ contradicting that $w$ and $\varphi(w)$ are coprime. Analogously, $p(\varphi(w))$ and $p(\varphi^2(w))$ cannot be zero.
By \cref{prop:crit_notinv_parameters}, we have $0=p(0)=p(w)+p(\varphi(w))+p(\varphi^2(w))$ so that each nonzero element of $\FF_2^2$ must appear exactly once in the sum on the right-hand side. 
Let $v\in S$ be the element with $p(v)=(1,0)$, thus $v$ is of the form $a^n+\sum_{i=1}^{n-1}\lambda_i a^ib^{n-i}$. Since the coefficient of the monomial $a^n$ in $\varphi(v) = b^n+\sum_{i=1}^{n-1}\lambda_ib^i(a+b)^{n-i}$ is zero and $p(w)\neq 0$ for all $w\in S$,
we obtain $p(\varphi(v))=(0,1)$. Finally, we have $p(\varphi^2(v))=p(v)+p(\varphi(v))=(1,1)$.
\end{proof}

\begin{remark}
    If we equip $\FF_2\oplus \FF_2$ with the $C_3$-action, where the generator $\varphi\in C_3$ sends $(1,0)$ to $(0,1)$ and $(0,1)$ to $(1,1)$, then \cref{lem:wloghighestsummands}
    states that the map $p$ is equivariant on all elements $v$ such that $v,\varphi(v)$ are coprime. It is not equivariant on all elements; for example, $p(ab)=(0,0)$ while $p(\varphi(ab))=p(b(a+b))=(0,1)$.
\end{remark}

We split the classification of the Steenrod closed parameter ideals generated by an orbit in two cases. We will use the \emph{first Steenrod square} $\Sq^1$, i.e., the linear map $H^n(BP)\to H^{n+1}(BP)$ for $n\geq 0$ obtained from the total Steenrod square by restriction to the degree $n$ part and projection to the degree $n+1$ part. We have $\Sq^1(z) =z^2$ for $z\in H^1(BP)$ and $\Sq^1$ is a derivation as a linear map $H^*(BP)\to H^*(BP)$.

\begin{lemma}\label{lem:steenrododd}
    Let $n$ be odd. An orbit $S\subset H^n(BP)$ generates a Steenrod closed parameter ideal if and only if $n=1$ and $S=C_3a$.
\end{lemma}
\begin{proof}  The orbit $C_3a$ generates a Steenrod closed parameter ideal by \cref{lem:poweroftwo_orbit}.

Conversely, suppose that an orbit $S$ generates a Steenrod closed parameter ideal. By \cref{lem:wloghighestsummands} there exists a unique $v \in S$ of the form $v=a^n+\sum_{i=1}^{n-1} \mu_i a^{n-i}b^i$. Moreover, the monomial $b^n$ appears in $\varphi(v)$, while $a^n$ does not.

Writing $v=ax^2+by^2$, we have $\Sq^1(v)=a^2x^2+b^2y^2$. By the choice of $v$, the coefficient of the monomial $a^{n+1}$ in $\Sq^1(v)$ is one, while the coefficient of $b^{n+1}$ is zero. Since the ideal generated by $S$ is Steenrod closed, we can write
\begin{align*}
\Sq^1(v)&=(\lambda_1 a+\lambda_2 b)v+(\nu_1a+\nu_2b)\varphi(v)\,
\end{align*}
with $\lambda_1,\lambda_2,\nu_1,\nu_2\in \FF_2$.
The observations about the monomials imply that $\lambda_1=1$ and $\nu_2=0$. Inserting $\varphi(v)= b\varphi(x)^2+(a+b)\varphi(y)^2$, we obtain
\begin{align*}
a^2x^2+b^2y^2=\Sq^1(v)=&(a+\lambda_2 b)v+(\nu_1a)\varphi(v)\\
=&(a+\lambda_2 b)(ax^2+by^2)+(\nu_1a)(b(\varphi(x)^2+\varphi(y)^2)+a\varphi(y)^2)\\
=&(a^2x^2+\lambda_2b^2y^2+\nu_1a^2\varphi(y^2))\\
&+ab(y^2+\lambda_2x^2+\nu_1(\varphi(x)^2+\varphi(y)^2))\,.
\end{align*}
Note that $\lambda_2^2=\lambda_2$ and $\nu_1^2=\nu_1$ since both elements are either zero or one. We apply the Kameko maps $\kappa_{1}$ and $\kappa_{ab}$ to deduce the equations
\begin{align}
\label{eq:oddkappa0}by &=\lambda_2by+\nu_1a\varphi(y)\\
\label{eq:oddkappaab}0&=y+\lambda_2x+\nu_1(\varphi(x)+\varphi(y))\,.
\end{align}
Applying the Kameko map $\kappa_a$ to $v+\varphi(v)+\varphi^2(v)=0$, we get
\begin{equation}\label{eq:noddeq1}
0=x+\varphi(y)+\varphi^2(x)+\varphi^2(y)\,.
\end{equation}
Now, we consider the four possible cases for the pair $(\lambda_2,\nu_1)$.
\begin{enumerate}[leftmargin=*]
    \item If $\lambda_2=0$, $\nu_1=0$, then $y=0$ by \eqref{eq:oddkappaab} and thus $v=ax^2$. By \eqref{eq:noddeq1}, $x=\varphi^2(x)$ and so $x$ is $\varphi$-invariant. Since $1=\gcd(v,\varphi(v))=\gcd(ax^2,\varphi(ax^2))=\gcd(ax^2,bx^2)$, it follows that $x=1$, i.e., $v=a$.
    \item If $\lambda_2=1$, $\nu_1=1$, then $y=0$ by \eqref{eq:oddkappa0} and it follows that $v=a$ as before.
    \item If $\lambda_2=0$, $\nu_1=1$, then $0=y+\varphi(x)+\varphi(y)$ by \eqref{eq:oddkappaab}. Applying $\varphi$ to this equation together with \eqref{eq:noddeq1} implies that $x=0$, contradicting that $a^n$ is a monomial in $v$. 
    \item If $\lambda_2=1$, $\nu_1=0$, then \eqref{eq:oddkappaab} shows that $x=y$ and thus $x=\varphi(x)$ by \eqref{eq:noddeq1}. Since $1=\gcd(v,\varphi(v))=\gcd((a+b)x^2,ax^2)$, we get $x=1$ and thus $v=a+b$. This contradicts the assumption that the monomial $b$ does not appear in $v$.
\end{enumerate}
We have shown that $v=a$ and hence $S=C_3a$ as claimed.
\end{proof}

\begin{lemma}\label{lem:steenrodeven}
    Let $n=2m$ be even. Then an orbit $S\subset H^n(BP)$ generates a Steenrod closed parameter ideal if and only if $S=T^{[2]}$ for an orbit $T\subset H^m(BP)$ that generates a Steenrod closed parameter ideal.
\end{lemma}
\begin{proof} If $T$ generates a Steenrod closed parameter ideal, so does $T^{[2]}$ by \cref{lem:interestingsquared}. Conversely, assume that an orbit $S\subset H^n(BP)$ generates a Steenrod closed parameter ideal. By \cref{lem:wloghighestsummands} there is a unique $v \in S$ of the form $v=a^n+\sum_{i=1}^{n-1} \mu_i a^{n-i}b^i$. Moreover, the monomial $b^n$ appears in $\varphi(v)$, while $a^n$ does not.

Writing $v=x^2+aby^2$, we have $\Sq^1(v)=(a^2b+ab^2)y^2$. We want to conclude that $y=0$. Since the ideal generated by $S$ is Steenrod closed, we can write
\begin{align*}
(a^2b+ab^2)y^2=\Sq^1(v)=(\lambda_1a+\lambda_2b)v+(\nu_1a+\nu_2b)\varphi(v)
\end{align*}
with $\lambda_1,\lambda_2,\nu_1,\nu_2\in \FF_2$. Since $a^{n+1}$ and $b^{n+1}$ do not appear on the left-hand side, we get $\lambda_1=0$ and $\nu_2=0$. Inserting $\varphi(v)=\varphi(x)^2+b(a+b)\varphi(y)^2$, we obtain
$$ 
(a^2b+ab^2)y^2=\lambda_2b(x^2+aby^2)+\nu_1a(\varphi(x)^2+b(a+b)\varphi(y)^2)\,.
$$
Using that $\lambda_2=\lambda_2^2$ and $\nu_1=\nu_1^2$, we apply $\kappa_a$ and $\kappa_b$ to deduce the equations
\begin{align}
\label{eq:evenkappaa}    by&=\lambda_2by+\nu_1\varphi(x)+\nu_1b \varphi(y)\\
\label{eq:evenkappab}    ay&=\lambda_2x+\nu_1a\varphi(y)\,.
\end{align}

We consider four cases:
\begin{enumerate}[leftmargin=*]
\item If $\lambda_2=0$, $\nu_1=0$,
then \eqref{eq:evenkappaa} simplifies to $by=0$. Hence $y=0$ and thus $v=x^2$.
\item If $\lambda_2=0$, $\nu_1=1$,
then $ay=a\varphi(y)$ by \eqref{eq:evenkappab}. Hence $y$ is $\varphi$-invariant. It follows from \eqref{eq:evenkappaa} that $x=0$ contradicting that $a^n$ appears in $v$.
\item If $\lambda_2=1$, $\nu_1=0$,
then $x=ay$ by \eqref{eq:evenkappab}. Thus $v=a(a+b)y^2$ and $\varphi(v) = ba\varphi(y)^2$ are divisible by $a$ contradicting that $v$ and $\varphi(v)$ are coprime.
\item If $\lambda_2=1$, $\nu_1=1$,
then $0=\varphi(x)+b\varphi(y)=\varphi(x+ay)$ by \eqref{eq:evenkappaa} and hence $x=ay$. It follows from \eqref{eq:evenkappab} that $0=y=x$. Thus $v=0$ contradicting that $S$ generates a parameter ideal.
\end{enumerate}
We have shown that $v=x^2$. Then $T=C_3x$ satisfies $T^{[2]}=S$ and generates a Steenrod closed parameter ideal by \cref{lem:interestingsquared}.
\end{proof}
Summarizing, we classify the Steenrod closed parameter ideals generated by an orbit.
\begin{theorem}\label{thm:classification_ideals_from_orbits}
Suppose that an orbit $S\subset H^*(BP)$ generates a parameter ideal. Then this ideal is Steenrod closed if and only if $C_3x=C_3a^n$ where $n$ is a power of two.
\end{theorem}
\begin{proof}
If $n$ is a power of two, then the $C_3$-orbit of $a^n$ generates a Steenrod closed parameter ideal by \cref{lem:poweroftwo_orbit}. 
The converse follows from \cref{lem:steenrodeven} and \cref{lem:steenrododd}.
\end{proof}

\section{\texorpdfstring{Applications to $A_4$-actions on finite CW complexes}{Applications to A4-actions on finite CW complexes}}
In this section we provide topological applications of the algebraic results from \cref{sec:steenrod_closed_parameter_ideals}.

\begin{theorem}\label{thm:actionspaces}
    Let $X$ be a finite $A_4$-CW complex with four dimensional cohomology with a basis of degrees $0\leq n\leq m\leq d$ such that the restricted $P$-action is free. Let $J$ denote the kernel of $H^*(BP)\to H^*(X/P)$. Then the following statements are equivalent:
    \begin{enumerate}
        \item\label{it:actionspacesi} The $C_3$-action on $H^*(X)$ is nontrivial;
        \item\label{it:actionspacesii} $J$ is generated by a $C_3$-orbit of an element in $H^n(BP)$;
        \item\label{it:actionspacesiii} $n=m$;
        \item\label{it:actionspacesiv} $J=\langle a^{n+1}, b^{n+1}\rangle$ and $n+1$ is a power of two;
        \item\label{it:actionspcesv} the map $H^*(BP)\to H^*(X/P)$ induces an isomorphism \[\FF_2[a,b]/\langle a^{n+1},b^{n+1}\rangle\cong H^*(X/P)\]
        and $n+1$ is a power of two.
    \end{enumerate}
\end{theorem}\begin{proof} 
The map $H^*(BP)\to H^*(X/P)$ is surjective by \cref{prop:Steenrodclosedparameteridealfromaction}, hence \eqref{it:actionspacesiv} and \eqref{it:actionspcesv} are equivalent.
By \cref{prop:Steenrodclosedparameteridealfromaction}, the ideal $J$ is a Steenrod closed $C_3$-invariant parameter ideal. We show that \eqref{it:actionspacesii} is equivalent to all of \eqref{it:actionspacesi}-\eqref{it:actionspacesiv}. The statements \eqref{it:actionspacesi} and \eqref{it:actionspacesii} are equivalent by \cref{prop:crit_inv_parameters} and \cref{prop:crit_notinv_parameters}. The parameters of $J$ have degrees $n+1$ and $m+1$ by \cite[Proposition~4.3]{ruepingstephan2024}. Thus, \eqref{it:actionspacesii} is equivalent to \eqref{it:actionspacesiii} by \cref{lem:parameterssamedegree}. Finally, \eqref{it:actionspacesii} is equivalent to \eqref{it:actionspacesiv} by \cref{thm:classification_ideals_from_orbits}.    
\end{proof}

Joint work with Erg{\" u}n Yal{\c c}{\i }n provides extensive restrictions for the possible degrees $(m,n)$ such that there exists a finite CW complex $X\simeq S^m\times S^n$ with a free $A_4$-action; see \cite[Corollary~7.9]{ruepingstephanyalcin2022}. For more general actions we have the following restrictions.
\begin{theorem}\label{thm:degreesofparameters}
    Let $X$ be a finite $A_4$-CW complex such that $H^*(X;\FF_2)$ is four-dimensional with homogeneous basis elements in degrees $0\leq m\leq n\leq d$.  If the restricted $P$-action is free, then the unordered pair $m+1,n+1$ is from the following list:
    \begin{enumerate}
    \item \label{it:degreesofparametersi} $(3k,2l)$ with $l\ge 1$ and $1\le k\le 2^t$, where $2^t$ is the largest power of $2$ dividing $l$, 
\item \label{it:degreesofparametersii} $(3i+2^{s+r+1}-2^{s+1},2^{s+r+1}-2^s)$
for $s\ge 0$, $r \ge 1$ and $0\le i<2^{s-1}$, or
\item \label{it:degreesofparametersiii}
$(2^j,2^j)$ for $j\geq 0$.
\end{enumerate}
\end{theorem}
\begin{proof}
    The kernel of $H^*(BP)\to H^*(X/P)$ is a $C_3$-invariant Steenrod closed parameter ideal with parameters of degrees $m+1$, $n+1$ by \cref{prop:Steenrodclosedparameteridealfromaction}. If the parameter ideal satisfies the equivalent conditions of \cref{prop:crit_inv_parameters}, then the degrees of the parameters must be from \eqref{it:degreesofparametersi} or \eqref{it:degreesofparametersii} by \cite[Corollary~6.13]{ruepingstephanyalcin2022}. Otherwise, the parameter ideal satisfies the equivalent conditions from \cref{prop:crit_notinv_parameters} and thus has parameters of degrees $(2^j, 2^j)$ by \cref{thm:classification_ideals_from_orbits}.
\end{proof}

\begin{corollary}\label{cor:percentageofpairs}
    For $r\ge 0$, the percentage of those pairs $(m,n)$ with $0\le m,n\le r$ such that there exists a finite $P$-free, $A_4$-CW complex whose mod $2$ cohomology is four-dimensional with homogeneous basis elements in degrees $0,m,n,d$ for some $d\ge m,n$ tends to zero as $r\to \infty$.
\end{corollary}
\begin{proof} This follows from \cref{thm:degreesofparameters} using \cite[Lemma~6.15]{ruepingstephanyalcin2022} and that the number of pairs of the form \eqref{it:degreesofparametersiii} grows like $\log(r)$.
\end{proof}

In \cite[Theorem~1.3]{ruepingstephan2024} we showed that
cochain complexes of finite, free $A_4$-CW complexes
with four-dimensional cohomology are rigid: They depend only on the degrees of the nonzero cohomology groups.  Here we extend this result to $P$-free $A_4$-CW complexes.

\begin{theorem}\label{thm:rigid} Suppose that $X$ and $Y$ are two finite, $P$-free, $A_4$-CW complexes with four-dimensional cohomologies. If the graded vector spaces $H^*(X)$ and $H^*(Y)$ are isomorphic, then their cellular cochain complexes $C^*(X)$ and $C^*(Y)$ are homotopy equivalent over $\FF_2[A_4]$.    
\end{theorem}
\begin{proof}
    The cellular cochain complexes are bounded and consist of finitely generated projective $\FF_2[A_4]$-modules, that is, they are perfect. Note that the smallest degree with nonzero homology is zero. 
    Thus, by \cref{thm:classification}, the two chain complexes are homotopy equivalent if and only if their annihilator ideals in $H^*(BP)$ agree. 

    By \cref{prop:Steenrodclosedparameteridealfromaction}, these ideals are Steenrod closed $C_3$-invariant parameter ideals. The degrees of the parameters only depend on the cohomology by \cite[Proposition~4.3]{ruepingstephan2024}. Thus, the degrees of the parameters for $X$ agree with the degrees of the parameters for $Y$ if $H^*(X)\cong H^*(Y)$. It remains to show that there is at most one $C_3$-invariant Steenrod closed parameter ideal with parameters of two given degrees. If the two degrees are the same, then there is only one such ideal by \cref{thm:actionspaces}.
    
    If the two degrees are different, then the ideals are not generated by an orbit. Instead, by \cref{prop:crit_notinv_parameters}, they are $C_3$-invariant Steenrod closed parameter ideals that satisfy the equivalent conditions from \cref{prop:crit_inv_parameters}. These ideals are extensions of Steenrod closed parameter ideals in $H^*(BA_4)$ by \cref{rem:bijectionextensionrestriction}. Since parameters for an ideal in $H^*(BA_4)$ also form a system of parameters of its extension in $H^*(BP)$, it suffices to establish uniqueness of Steenrod closed parameter ideals with parameters of given degrees in $H^*(BA_4)$. This holds by \cite[Corollary~6.13]{ruepingstephanyalcin2022}.
\end{proof}

\begin{remark} The cellular cochain complexes of finite, free $P$-CW complexes $X$ with four-dimensional cohomologies are not determined by $H^*(X)$. For instance, under the classification \cref{thm:classification}, 
the cellular cochain complex of the product action of the antipodal actions on $S^2$ and $S^3$ corresponds to the parameter ideal $\langle a^3,b^4\rangle$ whereas the free $P$-action on $S^2\times S^3$ given by restricting a free $A_4$-action corresponds to $\langle a^2b+ab^2, a^4+a^2b^2+b^4 \rangle$.
\end{remark}

Next we determine for which $n$ the ideal $J=\langle a^{n+1}, b^{n+1}\rangle$ from \cref{thm:actionspaces} can be realized by a $P$-free $A_4$-action on a product of two spheres $X=S^n\times S^n$.
\begin{theorem}\label{thm:oddaction0137} 
If $A_4$ acts on a finite CW complex homotopy equivalent to $S^n\times S^n$ such that the action of its $2$-Sylow subgroup $P=\ZZ/2\times \ZZ/2$ is free, then $n\in\{0,1,3,7\}$. Conversely, for each $n\in\{0,1,3,7\}$ there exists such an $A_4$-action on $S^n\times S^n$.
\end{theorem}

\begin{proof} Given such an $A_4$-action on a finite CW complex $X\simeq S^n\times S^n$, the $C_3$-action on the two-dimensional vector space $H^n(X;\FF_2)$ is nontrivial by \cref{thm:actionspaces}. Thus all three nonzero elements must lie in the same orbit. In particular, $H^n(X;\FF_2)$ does not have a basis consisting of $C_3$-orbits, i.e., is not a permutation module. It follows that $n\in \{0,1,3,7\}$ by a result of Schultz \cite[p.~139]{schultz1971} (see also \cite[Lemma~5.1]{adembrowder1988}).

Conversely, for $n=0,1,3,7$, we can use the multiplications in $\mathbb{R},\mathbb{C},\mathbb{H},\mathbb{O}$ to extend the free diagonal $P$-action on $S^n\times S^n$ that multiplies with $\pm 1$ on each factor to an $A_4$-action; 
see \cite[Proof of Theorem~3.4]{blaszczyk2013free}, \cite{plaktha1996}.
\end{proof}

\begin{question}
Which ideals $J=\langle a^{n+1}, b^{n+1}\rangle$ can be realized by a finite $A_4$-CW complex as in \cref{thm:actionspaces} for $n=2^k-1$ with $k>3$?    
\end{question}

\bibliographystyle{amsalpha}
\bibliography{bibl}
\end{document}